\documentclass{article} 
\usepackage{graphicx, amsmath, amssymb}
\usepackage{comment} 
\usepackage{xspace}
\usepackage{amsthm}

\theoremstyle{plain}
\newtheorem{theorem}{Theorem}[section]
\newtheorem{lemma}[theorem]{Lemma}

\newtheorem{definition}[theorem]{Definition}

\theoremstyle{remark}

\theoremstyle{question}

\newcommand{\il}{{\ensuremath{\textup{\textbf{IL}}}}\xspace}
\newcommand{\extil}[1]{\ensuremath{\textup{\textbf{IL}}{\sf\ensuremath{#1}}}\xspace}
\newcommand{\gl}{{\ensuremath{\textup{\textbf{GL}}}}\xspace}

\newcommand{\K}{{\ensuremath{\textup{\textbf{K}}}}\xspace}
\newcommand{\Kfour}{{\ensuremath{\textup{\textbf{K}}}\mathbf{4}}\xspace}

\begin{document}

  \title{The closed fragment of \il is PSPACE hard} 
  \author{F\'elix Bou and Joost J. Joosten} 
  \date{2011}
\maketitle

\begin{abstract} 
In this paper we consider $\il_0$, the closed fragment of the basic interpretability logic \il. We show that we can translate $\gl_1$, the one variable fragment of G\"odel-L\"ob's provabilty logic \gl, into $\il_0$. Invoking a result on the PSPACE completeness of $\gl_1$ we obtain the PSPACE hardness of $\il_0$. 
%This, in particular entails that there are no simple normal forms for formulas in $\il_0$.
\end{abstract}

\section{Introduction}

For a propositional logic $\sf L$, the closed fragment --we write ${\sf L}_0$-- of that logic consists of those theorems of $\sf L$ that do not contain any propositional variables at all. For various logics, it is known that the closed fragment is a lot easier than the full logic itself. The simplicity of the closed fragment can be captured by the complexity class of a decision procedure of theoremhood. Moreover, in all cases where $L_0$ is known to be simpler than $L$ in this sense, we have a set of \emph{normal forms} for ${\sf L}_0$ and a normal form theorem to the effect that each closed formula can be written in a unique way as a special combination of normal form formulas.

Perhaps the most canonical example of this phenomenon is classical propositional
logic. Theoremhood in classical propositional logic is known to be co-NP complete
whereas the closed fragment is decidable in LOG-time. In this case, by
definition, the only two formulas in normal form are $\top$ and $\bot$.
For various modal logics the situation is similar but slightly
different. For the provability logic \gl, theoremhood is known to be a
PSPACE complete problem (see \cite[Theorem~18.29]{ChZa97}), whereas
provability of formulas in the closed fragment is known to be
PTIME decidable (see \cite[Theorem 9]{ChaRy03}). Moreover, the normal
form theorem (\cite[Chapter~7]{Boo93}) states that each formula in the
closed fragment is provably equivalent to a Boolean combination of
formulas of the form $\Box^n \bot$ with $n\in \omega$. 

Interpretability logics arise as natural extensions of \gl.
The logic \gl has only one modal operator $\Box A$ to capture that ``$A$
is provable in some basic theory $T$''. Interpretability logics have an
additional binary modality $A\rhd B$ to capture that ``the theory $T +A$
interprets the theory $T+B$''.

These interpretability logics are always defined as some core part \il as defined below, together with some additional principles. However, as soon as the additional principles prove some rather weak principle $\sf F$, the technical details of which are irrelevant for the moment, then closed interpretability formulas can be expressed without the modality $\rhd$ and the normal forms are the same as those of \gl: Boolean combination of formulas of the form $\Box^n \bot$ with $n\in \omega$ (see \cite{HavS91}). It is good to stress here that all interpretability logics with some interesting meta-mathematical content do contain the principle $\sf F$. For logics below \extil{F} and in particular for \il itself, it is not known if there exists a natural set of normal forms. 
%In \cite{Vukovic2011} it is shown that in the closed fragment of the logic \il the modality $\rhd$ cannot be dispensed with.

Not for all modal logics it is the case that the $L_0$ is simpler than
$L$. In particular, it is known that the minimal modal logic $\K$ and
its closed fragment $\K_0$ are both PSPACE complete (see
\cite[Corollary~4]{ChaRy03}). The same also happens for the modal logic
$\Kfour$ of transitive frames (see also \cite{ChaRy03}).

We shall see in this paper that the logic \il is like these logics $\K$ and $\Kfour$ in that also the closed fragment of \il is PSPACE hard thereby settling an open question in \cite{Visser1997} in the negative as to whether the closed fragment allows a nice characterization.

\section{Interpretability logics}

Interpretability logics have been primarily used to study in a
formalized setting the notion of relativized interpretability which is
captured by a binary modal operator $\rhd$. The phrase $p\rhd q$ is to
be read as ``($T$ together with the translation of $p$) interprets ($T$
together with the translation of $q$)" for some base theory $T$.
Different theories $T$ prove different modal principles to hold.
However, all theories that allow for coding of syntax and thus for
formalizing the notion of interpretability do validate some core logic
which is called \il.

\subsection{The logic \il}
We recall that \gl is the normal modal logic with one modality $\Box$ whose non-logical axioms are instantiations of the following axiom schemes.

\begin{enumerate}
\item 
$\Box(A\to B) \to (\Box A \to \Box B)$

\item
$\Box (\Box A \to A) \to \Box A$
\end{enumerate}
It is well known that \gl proves the transitivity axiom, that is,
\[
\Box A \to \Box \Box A.
\]
The logic \il is formulated in a propositional modal logic with two
modalities $\Box$ and $\rhd$. We shall use the following reading
conventions. The strongest binding operators are $\neg$ and $\Box$
followed by $\vee$ and $\wedge$ which in turn bind stronger than
$\rhd$. The weakest binding connectives are the implications $\leftrightarrow$ and $\to$. We shall write $\Diamond \varphi$ as shorthand for $\neg \Box \neg \varphi$.

\begin{definition}
The logic \il is a normal modal logic containing \gl whose rules are Modus Ponens and Necessitation and whose axioms other than all propositional tautologies are the instances of the following axiom schemes.

\begin{description}
\item[$J1$]
$\Box (A \to B) \to A\rhd B$

\item[$J2$]
$(A \rhd B) \wedge (B\rhd C) \to A\rhd C$

\item[$J3$]
$(A\rhd C) \wedge (B\rhd C) \to A\vee B \rhd C$

\item[$J4$]
$A\rhd B \to (\Diamond A \to \Diamond B)$

\item[$J5$]
$\Diamond A \rhd A$
\end{description}

\end{definition}
It follows from $J1$ and $J4$ that $\Box$ is expressible in terms of $\rhd$ within \il:
\[
\il \vdash \Box A \leftrightarrow \neg A \rhd \bot.
\]
The logic \extil{F} is obtained by adding the axiom $\sf F$ to \il.
\[
{\sf F} \ := \ \ \ \ \Diamond A \to \neg(A\rhd \Diamond A)
\] 
This principle can be seen as a natural generalization of G\"odel's second incompleteness theorem. G\"odel's second incompleteness theorem states that any recursive theory, whenever consistent, does not prove its own consistency. The principle  $\sf F$ states that any recursive theory, whenever consistent, does not even \emph{interpret} its own consistency.
\subsection{Semantics for \il}
The logic \il allows for natural Kripke semantics where the binary modality $\rhd$ is modeled by a ternary relation. Rather than working with a ternary relation, we tend to conceive the semantics for $\rhd$ as a collection of binary relations.

\begin{definition} 
An \il model, also called Veltman-model, is a quadruple $\langle W, R,
\{ S_x : x \in W \} , \Vdash \rangle$ where $W$ is a non-empty set of worlds, $R$ is a binary relation on $W$ that is transitive and conversely well-founded. For each $x\in W$, the binary relation $S_x$ is transitive and reflexive such that moreover
\begin{enumerate}
\item 
$yS_xz\to xRy \wedge xRz$;

\item
$xRyRz \to yS_xz$.

\end{enumerate}
The relation $\Vdash$ is a usual forcing relation that can be conceived as a map assigning to each propositional variable $p$ a subset $v(p)$ of $W$ of the worlds where $p$ holds. We write $x\Vdash p$ to indicate that $x\in v(p)$. The relation $\Vdash$ is extended to the set of all formulae by stipulating that
\begin{enumerate}
\item 
$x\Vdash \Box A \ \ \Leftrightarrow \ \ \forall y \ (xRy \to y \Vdash A)$;

\item 
$x\Vdash  A\rhd B \ \ \Leftrightarrow \ \ \forall y (xRy \wedge y \Vdash A \to \exists z (yS_xz \wedge z\Vdash B))$.

\end{enumerate}

\end{definition}
It is well-known that \il is sound and complete with respect to the class of all Veltman models (see \cite{deJonghJaparidze1998}).

\subsection{Fragments}
We shall denote by $\il_0$ the fragment of \il that consists of those modal formulae provable in \il that contain no propositional variables. Likewise, by $\gl_1$ we shall denote those formulas in the language of \gl that contain only one variable and are theorems of \gl.

\section{Translating $\gl_1$ into $\il_0$}

Let $p$ be the variable of $\gl_1$. We shall translate  this variable to some formula in the closed fragment of \il that essentially uses the $\rhd$ modality. It is easy to see that such formulas exist. Examples are given in \cite{Visser1997} (Section 5.4) and in \cite{VucovicTA}. The formula that we use here is equal to the one exposed in \cite{VucovicTA}.

\subsection{Some motivation for our translation}
In this section we shall expose a translation that reduces theoremhood of $\gl_1$ to $\il_0$ thereby establishing PSPACE hardness of the latter. The motivation for this translation is mainly semantical.

We will code the information as to whether $p$ holds or not in a world $x$ by making the formula $\top \rhd \Diamond \top$ true at $x$ if and only if $x\Vdash p$. To this extent we can glue to each $x$ two new worlds $x_1$ and $x_2$ with $xRx_1Rx_2$ and\footnote{We should add some more $S_x$ relations too on the already existing part of the model. For the motivational part here, we just focus on the newly added worlds $x_1$ and $x_2$.} $x_2S_xx_1 \ \Leftrightarrow \ x\Vdash p$. For this, all the $S_x$ relations in \il are sufficiently independent. This idea should motivate why we translate $p$ to $\Diamond \Diamond \top \to \top \rhd \Diamond \top$.

Moreover, with this approach the points that we are interested in, that is, the original points, become easily definable by the formula $\Diamond \Diamond \top$. Thus, when quantifying over points that we are interested in, we should relativize to our old domain. This explains why we shall translate $\Box A$ to $\Box(\Diamond \Diamond \top \to A^{\dag})$ where $A^{\dag}$ is the translation of $A$.

We shall see in Subsection \ref{section:ConstructionsOnModels} that we do not actually need to glue so many different new worlds to code all the behavior of the $x\vdash p$ for all $x$ in the model. By transitivity it suffices to add some worlds only at the top of the model.

\subsection{The translation}

We consider the following translation $\dag$ of formulas of $\gl_1$ into formulas of $\il_0$:

\begin{enumerate}
\item 
$\bot^{\dag} = \bot$

\item
$p^{\dag} = \Diamond \Diamond \top \to (\top \rhd \Diamond \top)$

\item
$(A\to B)^{\dag} = A^{\dag} \to B^{\dag}$

\item
$(\Box A)^{\dag} = \Box ( \Diamond \Diamond \top \to A^{\dag})$

\end{enumerate}
 
\begin{lemma}\label{lemma:translationPreservesTruth}
Let $A$ be a formula of \gl that only contains the propositional variable $p$. If $\gl \vdash A$, then $\il \vdash A^{\dag}$.
\end{lemma}

%\begin{proof} %% Sintactic ONE
%So, suppose $\gl \vdash A$. We know that \gl has a cut-free proof $\pi$ (see system $\mathcal{G}^1$ in \cite{Lei81}) which thus satisfies the sub-formula property. Consequently, each sequent in $\pi$ contains at most the variable $p$. It is an easy check that proofs only containing $p$ are preserved under $\dag$.
%\end{proof}

\begin{proof} %% Semantic ONE
  So, suppose $\il \not \vdash A^{\dag}$. Then, there is an \il model
  $\mathcal{M} = \langle W, R, \{ S_x : x \in W \} , \Vdash \rangle$
  and a world $w \in W$ such that $\mathcal{M},w \not \Vdash A^{\dag}$.
  Next, we consider the \gl model $\mathcal{N} = \langle W', R', \Vdash'
  \rangle$ defined by:
  \begin{enumerate}
    \item $W':= \{ w \} \cup \{ x \in W: \mathcal{M},x \Vdash \Diamond
      \Diamond \top \}$,
    \item $R':= R \cap (W' \times W')$,
    \item $x \Vdash' p$  iff  $\mathcal{M},x \Vdash p^{\dag}$ (for
      every $x \in W'$).
  \end{enumerate}
  We point out that the union defining $W'$ may be a non-disjoint one.
  Using the definition of $\mathcal{N}$ it is straightforward to prove (by induction on
  the length of the formula) that for every formula $B$ which only contains
  the propositional variable $p$, 
  \begin{center}
    $\mathcal{N},x \Vdash' B$  iff  $\mathcal{M},x \Vdash B^{\dag}$
    \quad (for every $x \in W'$).
  \end{center}
  In particular, we get that $\mathcal{N},w \not \Vdash' A$. Therefore, 
  $\gl \not \vdash A$.
\end{proof}

Lemma \ref{lemma:translationPreservesTruth} is the easier direction of what we shall see is an equivalence. In particular, the lemma allows for an easy proof-theoretic proof.

\begin{proof}
So, suppose $\gl \vdash A$. We know that \gl has a cut-free proof $\pi$ (see system $\mathcal{G}^1$ in \cite{LeivantGL1981}) which thus satisfies the sub-formula property. Consequently, each sequent in $\pi$ contains at most the variable $p$. It is an easy check that proofs only containing $p$ are preserved under $\dag$.
\end{proof}

\subsection{Construction on models}\label{section:ConstructionsOnModels}

In this subsection we shall prove the converse to Lemma \ref{lemma:translationPreservesTruth}.

\begin{lemma}\label{lemma:translationAndModels}
Let $A$ be a formula of \gl that only contains the propositional variable $p$. If $\il \vdash A^{\dag}$, then $\gl \vdash A$.
\end{lemma}

\begin{proof}
By the completeness proofs of \il and \gl we know that it is sufficient to show that 
\begin{center}
$\forall A \ [\forall^{\mbox{\il-model}} \mathcal{M} \ \mathcal{M}
\models A^{\dag} \ \ \Longrightarrow \ \  \forall^{\mbox{\gl-model}}
\mathcal{N} \ \mathcal{N} \models A ]$,
\end{center}
or equivalently
\begin{center}
$\forall A \ [\exists^{\mbox{\gl-model}} \mathcal{N} \, \exists n {\in}
\mathcal{N} \ \mathcal{N}, n \Vdash A \ \ \Longrightarrow \ \
\exists^{\mbox{\il-model}} \mathcal{M} \, \exists m {\in} \mathcal{M}\
\mathcal{M},m \Vdash A^{\dag} ]$.
\end{center}
To this extent we shall exhibit a transformation on \gl models that
yields the desired \il model. Let $\mathcal{N}:= \langle W, R, \Vdash
\rangle$ be a \gl-model for some \gl-formula $A$ at most containing the
variable $p$. As \gl is complete with respect to finite tree-like
models we may indeed assume that $\mathcal{N}$ has is such a finite tree-like model.  By
E=$\{ e_1, \ldots, e_l \}$ we denote the set of end-points in
$\mathcal{N}$, which is of course finite. We consider two disjoint copies of $E$
(which are also disjoint with $W$),
namely $E^{\flat} = \{ e_1^{\flat}, \ldots, e_l^{\flat} \} $ and
$E^{\natural} = \{ e_1^{\natural}, \ldots, e_l^{\natural} \}$. The idea is to `glue' these additional points $e_i^{\flat}$ and $e_i^{\natural}$ as a little $R$ chain of length two above\footnote{W.r.t. the $R$-relation of course.} the end-points $e_i$ so that each old point in the model will satisfy $\Diamond \Diamond \top$.

As $R$ is conversely
well-founded it is the case that each $x\in M$ is $R^=$-below some $e_i$
(here $xR^=y$ is a shorthand for ``$xRy \vee x=y$''). 

Now, we consider an \il model $\mathcal{M}=\langle W', R', \{ S'_x : x \in W' \} 
\rangle$ satisfying:
\begin{enumerate}
  \item $W':= W \uplus  E^{\flat} \uplus E^{\natural}$,
  \item $R'$ is the transitive closure of %the binary relation 
    $R \cup \{ (e,e^{\flat}): e \in E \} \cup \{ (e^{\flat},
    e^{\natural}): e \in E \}$,
  \item $S'_{e^{\flat}} = \{ (e^{\natural}, e^{\natural}) \}$ (for
    every $e^{\flat} \in E^{\flat}$).
  \item $S'_{e^{\natural}} = \emptyset$ (for
    every $e^{\natural} \in E^{\natural}$).
  \item for every $x \in W$ it holds that
    \begin{center}
    $S'_x$ is the smallest transitive and reflexive relation on $\{y\mid xR'y\}$ that contains both $S_x$ and $R'$ restricted to  $\{y\mid xR'y\}$ such that moreover for every
    $e \in E$,\\
      $( e^{\natural} , e^{\flat} ) \in S'_x$ \quad iff \quad
      $\mathcal{N},x \Vdash p$ and $x R^= e$.
    \end{center} 
\end{enumerate}
It is very easy to check that there is a unique \il model satisfying these
conditions. We notice that in the definition of $W'$ we have used the
symbol $\uplus$ to emphasize that these unions are indeed disjoint ones,
and we have not introduced a valuation $\Vdash'$ in $\mathcal{M}$
because our purpose is only to evaluate closed formulas (like
$A^{\dag})$. 

\medskip
First of all we note that in $\mathcal{M}$ we can modally define the old points in $\mathcal{N}$
since $\{ x \in W': \mathcal{M},x \Vdash \Diamond \Diamond \top \} = W$. The next step is to prove 
that for every formula $B$ which
only contains the propositional variable $p$, 
\begin{center}
  $\mathcal{N},x \Vdash' B$  iff  $\mathcal{M},x \Vdash B^{\dag}$
  \quad (for every $x \in W$).
\end{center}
The proof of this claim proceeds by an induction on the length of $B$.
\begin{itemize}
  \item For $\top$ or $\bot$ the claim is vacuous.
 
 \item If $B=p$ we have that $p^{\dag} = \Diamond \Diamond \top \to
   \top\rhd \Diamond \top$. By construction, $\mathcal{M}, x \Vdash
   \Diamond \Diamond \top$.  By construction, in any $R$-successor of
   $x$ one can go by an $S_x$ transition to some $e_i^{\flat}$ where
   $\Diamond \top$ holds. Thus, indeed, $\mathcal{M},x\Vdash \top \rhd
   \Diamond \top$.
   
 \item On the other hand, if $\mathcal{N}, x\Vdash \neg p$, then again
   $\mathcal{M}, x \Vdash \Diamond \Diamond \top$. But in this case we
   can go via an $R$-transition to some $e_i^{\natural}$. By construction
   there is no $S_x$-transition from $e_i^{\natural}$ to any point where
   $\Diamond \top$ holds,  whence $\mathcal{M},x\not \Vdash \top \rhd
   \Diamond \top$.

 \item The proof of the claim is trivial for both the Boolean
   connectives and the modal operator $\Box$. 
\end{itemize}
Now that the claim is established the lemma follows immediately.
\end{proof}

\section{Computational complexity of $\il_0$}
First we obtain PSPACE hardness of \il.
\subsection{PSPACE hardness}

If we combine Lemma \ref{lemma:translationPreservesTruth} and Lemma
\ref{lemma:translationAndModels}
we see that we have a reduction of $\gl_1$ to $\il_0$. That is, for any formula $A$ with at most one variable $p$ we have that
\[
\gl \vdash A \ \ \Leftrightarrow \ \ \il \vdash A^{\dag}.
\]
As it is known that $\gl_1$ is PSPACE complete (see~\cite[Theorem~7]{ChaRy03} and \cite{Svejdar2003}) we obtain the main result of this paper.

\begin{theorem}
The computational complexity of $\il_0$ is PSPACE hard.
\end{theorem}

If in addition to this we would know that \il is in PSPACE we would obtain PSPACE completeness. It is commonly held that indeed the complexity of full \il is PSPACE-compleet, but up to now nobody has yet proven this. It came as a bit of a surprise to the authors to find out that actually no complexity results in the field of interpretability logics are known. Thus, this short note could well be the precursor to further investigations in various interpretability logics on very natural complexity questions in the otherwise mature field of interpretability logic.

\subsection{On a normal form theorem for $\il_0$}
The PSPACE completeness of $\il_0$ does a-priori not exclude the possibility of a normal form theorem of \il. It is even conceivable that there exists some easily recognizable class of normal forms for $\il_0$ so that each formula in the language of $\il_0$ is  equivalent to a small sized boolean combination of these normal forms. In such a case the normal forms themselves may be easy and even easily comparable but then for an arbitrary formula it still remains hard (PSPACE) to see actually \emph{what} combination of normal forms it is provably equivalent to. These observations render a normal form theorem for $\il_0$ --if it would exist-- useless for most practical purposes.

%In particular, this theorem settles in the negative a long-standing question on whether the closed fragment of \il allows for some amenable normal form. Any possible description of a  normal form can not be simple

%%%% Una vez este hecha la bibliografía hay que sustituir esta
%%%% sección sustituyéndolo por la lista de referencias que toc
%\bibliographystyle{plain}
%\bibliography{biblio}

\end{document}